\documentclass[11]{amsart}

\usepackage{amsfonts}
\usepackage{amscd}
\usepackage{amsmath, mathrsfs, amssymb}
\usepackage{amsthm}
\usepackage{setspace}
\usepackage{hyperref}
\usepackage{enumerate}
\usepackage{caption}

\theoremstyle{plain}

\newtheorem{thm}{Theorem}[section]
\newtheorem{prop}[thm]{Proposition}
\newtheorem{lem}[thm]{Lemma}
\newtheorem{cor}[thm]{Corollary}
\theoremstyle{definition}

\newtheorem{remark}[thm]{Remark}
\numberwithin{equation}{section}

\newcommand{\PP}{\mathbb{P}}  
\newcommand{\CC}{\mathbb{C}}  
\newcommand{\OO}{\mathcal{O}}  
\newcommand{\cC}{\mathcal C} 
\newcommand{\cE}{\mathcal E} 
\newcommand{\cI}{\mathcal I} 
  
\newcommand{\cP}{\mathcal P}  
\newcommand{\cR}{\mathcal R}  
\newcommand{\M}{\overline{M}}  

\begin{document}

\title{Counting differentials with fixed residues}

\author{Dawei Chen}

\author{Miguel Prado}

\address{Department of Mathematics, Boston College, Chestnut Hill, MA 02467, USA}
\email{dawei.chen@bc.edu, pradogod@bc.edu}

\thanks{Research partially supported by National Science Foundation Grant DMS-2001040 and Simons Travel Support for Mathematicians Grant 635235.}

\begin{abstract}
 We investigate the count of meromorphic differentials on the Riemann sphere possessing a single zero, multiple poles with prescribed orders, and fixed residues at each pole. Gendron and Tahar previously examined this problem with respect to general residues using flat geometry, while Sugiyama approached it from the perspective of fixed-point multipliers of polynomial maps in the case of simple poles. In our study, we employ intersection theory on compactified moduli spaces of differentials, enabling us to handle arbitrary residue conditions and provide a complete solution to this problem. We also determine interesting combinatorial properties of the solution formula.  
\end{abstract}

\maketitle

\setcounter{tocdepth}{1}
\tableofcontents

\section{Introduction}

Consider the tuples $\mu=(a,-b_1,\ldots,-b_n)\in \mathbb{Z}^{n+1}$ where $a\geq 0$, $b_1,\ldots,b_n>0$ with $n\geq 2$, and $\rho=(r_1,\ldots,r_n)\in \mathbb{C}^n$ such that
$$
    a-b_1-\cdots-b_n=-2 \quad {\rm and} \quad r_1+\cdots+r_n=0.
    $$
How many meromorphic differentials $\omega$ on $\PP^1$ possessing a single zero $z$ and $n$ distinct poles $p_1,\ldots, p_n$ satisfy that 
$${\rm div} (\omega)=az-b_1p_1-\cdots-b_np_n \quad {\rm and}\quad  {\rm Res}_{p_i} \omega=r_i$$ 
for $i=1,\ldots,n$? Namely, we want to count the number of meromorphic differentials of zero and pole type $\mu$ that have fixed residues given by $\rho$ at each pole. Denote this number by $N(\mu, \rho)$. We remark that in our setting the poles are labeled. 

This question was studied by Gendron and Tahar \cite{GT} in the setting of isoresidual fibration. Assigning to $\omega$ the tuple of its residues at the poles induces a map from the moduli space of differentials of type $\mu$ to $\CC^{n} = \{(r_1,\ldots, r_n)\}$ where the image of the map dominates the hyperplane parameterizing $r_1 + \cdots + r_n = 0$ due to the Residue Theorem. Then $N(\mu, \rho)$ is equal to the fiber cardinality of the map over $\rho$. The value of $N(\mu, \rho)$ depends on whether the entries of $\rho$ satisfy additional partial sum vanishings. If there exists a nonempty proper subset of indices $\emptyset\neq I\subsetneq\{1,\ldots, n\}$ such that 
$$\sum_{i\in I} r_i=0,$$ 
then we say that $\rho$ satisfies a partial sum vanishing indexed by $I$. Due to the total sum vanishing, $I$ and $I^c$ index the same partial sum vanishing condition. If $\rho$ does not satisfy any partial sum vanishings, we say that $\rho$ is general. 

To express the formulas compactly, we will frequently use the function
\begin{equation}
\label{eqn:s}
    f(a,n) = \frac{a!}{(a-(n-2))!}. 
\end{equation} 
For a subset $I\subset \{1,\ldots, n\}$, we will also use the notation 
$$ b_I  = \sum_{i\in I} b_i. $$

In \cite{GT} the following cases were resolved by counting certain graphs that arise from the flat geometric structure of meromorphic differentials. 

\begin{thm}[{\cite[Theorem 1.2 and Proposition 1.3]{GT}}]
\label{thm:intro}
\begin{enumerate}
    \item[{\rm (i)}] For general $\rho$ that does not satisfy any partial sum vanishings, we have 
    \begin{equation*}
      N(\mu,\rho)=f(a,n).  
    \end{equation*}
    \item[{\rm (ii)}] If $\rho$ satisfies exactly one partial sum vanishing indexed by $I$, then 
    \begin{align*}
       N(\mu,\rho) &=f(a,n)-f(b_I -1,|I|+1)\cdot f(a-b_I + 1,n-|I|+1). 
    \end{align*}
    \end{enumerate}
\end{thm}

In this paper we study this question by using intersection theory on compactified moduli spaces of differentials and provide a complete solution for arbitrarily given residues as follows. 

\begin{thm}
\label{thm:main}
Let $\mu=(a,-b_1,\dots,-b_n)$ and $\rho=(r_1,\ldots,r_n)$ as defined above. Then the number $N(\mu,\rho)$ of meromorphic differentials on $\PP^1$ of type $\mu$ with 
 residues given by $\rho$ is 
\begin{equation}
\label{eqn:main}
    N(\mu,\rho) = \sum_{s=1}^S (-1)^{s-1}(a+1)^{s-2}\sum_{\mathcal{J}_s\in \mathcal{C}_s} \prod_{J_j\in \mathcal{J}_s} f(b_{J_j}-1,|J_j|+1)
\end{equation}               
where the set $\mathcal{C}_s$ consists of collections $\mathcal{J}_s=\{J_1,\ldots,J_s\}$ such that $J_1, \ldots, J_s$ are nonempty  disjoint subsets that partition 
$\{1,\ldots,n\}$ and $\sum_{j\in J_i} r_j = 0$ for $i = 1,\ldots, s$, and $S$ is the maximum value of such $s$. 
\end{thm}

Although the formula \eqref{eqn:main} looks bulky, it has a simple combinatorial structure. It begins with the number $f(a,n)$ for the case of general residues in Theorem~\ref{thm:intro} (i), next subtracts the sum of products of such numbers as in (ii) over all pairs $\{J, J^c\}$ where $J$ is a partial sum vanishing satisfied by $\rho$, then adds back the sum of products of such numbers over all triples $\{J_1, J_2, J_3\}$ where the $J_i$ form a disjoint union of $\{1,\ldots,n\}$ and are partial sum vanishings satisfied by $\rho$, and so on. These alternating summands are further multiplied by twists of associated level graphs (see~\eqref{eqn:twist}) where these twisting factors arise from the number of positive horizontal directions at the zero of a differential (i.e., the number of prongs as defined in \cite[Section 5.4]{BCGGM3}). The structure of the formula seems to suggest a heuristic argument by specializing from a general fiber to a special fiber in the isoresidual fibration and then applying the inclusion--exclusion principle. However, the powers of $a+1$ in the formula do not have a natural explanation this way, which seems to purely come from combinatorial computations. Therefore, our actual proof takes a different approach by establishing a delicate recursion of intersection numbers and then simplifying it to the desired expression. In the course of the proof we also obtain some interesting combinatorial properties of $N(\mu, \rho)$ as follows. 
\begin{cor}
\label{cor:numerical}
For $\mu = (a, -b_1,\ldots, -b_n)$ and any residue tuple $\rho$ that contains a nonzero entry, $N(\mu, \rho)$ is a polynomial of top degree $n-2$ in the variables $b_1,\ldots, b_n$. Moreover, for two residue tuples $\rho$ and $\rho'$ where every partial sum vanishing of $\rho$ holds for $\rho'$ but not vice versa, we have $N(\mu, \rho) > N(\mu, \rho')$. In particular, $N(\mu, \rho) = 0$ if and only if $\rho$ is identically zero.  
\end{cor}

Intuitively speaking, the above inequality says that as the residue tuple $\rho$ specializes to satisfy more partial sum vanishings, some fiber points of the isoresidual map approach to the boundary of the moduli space, so that they need to be removed from the count. 

The connection to intersection theory is that a partial sum vanishing condition of the residues corresponds to certain divisor class in the compactified moduli space of differentials. Then one expects that the intersection number of those divisor classes is equal to the number of differentials satisfying the given residue conditions. However, when the given residues are not general, there can be excess intersection in the boundary of the moduli space, which is the main technical difficulty that we will analyze and resolve. 
Moreover, despite that Theorem~\ref{thm:intro} (i) and (ii) are sub-cases of \eqref{eqn:main}, we will use them in our proof for arbitrary residues. Therefore, we will demonstrate the proof of Theorem~\ref{thm:intro} separately, which can also help the reader quickly get a feel about our method. 

The most special residue tuple is $\underline{0} = (0, \ldots, 0)$. Indeed it is known that a meromorphic differential on $\PP^1$ with a single zero and at least two poles cannot have zero residues at every pole (see~\cite[Lemma 3.6]{BCGGM1}), namely, $N(\mu, \underline{0}) = 0$. In this case our formula~\eqref{eqn:main} yields an interesting and nontrivial combinatorial identity, which we will prove directly in Lemma~\ref{lem:zero}. Given this fact, one can projectivize the isoresidual map by sending a differential modulo scaling to the projective tuple of its residues as a well-defined point in $\PP^{n-1}$, where the fiber cardinality over $[r_1:\cdots :r_n]$ still counts the same number $N(\mu, \rho)$ when regarding $\rho$ as the projective tuple $[r_1:\cdots :r_n]$. 

The second most special residue tuple is $(0,\ldots, 0, r, -r)$ where $r\neq 0$, for which the answer is also known and has a simple expression (see \cite[Proposition 3.1]{CC}, \cite[Proposition 4.6]{GT}, and \cite[Theorem 4.1]{BR}). For completeness we include a short proof in Proposition~\ref{prop:k=n-2} for this case, which also leads to an interesting combinatorial identity. 

In a different context, Sugiyama \cite{Su} studied this question from the viewpoint of polynomial maps on $\PP^1$ and related dynamics. 
Consider a polynomial $g(z)\in \mathbb{C}[z]$ with fixed points $p_1,\ldots, p_n$ each of which has multiplicity $b_i$, i.e., $p_i$ is a root of $g(z) = z$ with multiplicity $b_i$ where $b_1 + \cdots + b_n = \deg g$. The multiplier of $g$ at the fixed point $p_i$ is defined as 
$\lambda_i = g'(p_i)$, which determines the linear approximation of $g$ at $p_i$ and thus plays an important role in the dynamics of polynomial maps. 
 In order to relate to our setting, consider the differential $\omega = dz / (z - g(z))$. For $b_i = 1$, i.e., when $p_i$ is a simple fixed point of $g$ and a simple pole of $\omega$, the residue $r_i$ of $\omega$ at $p_i$ is $1/(1- \lambda_i)$. Therefore, determining polynomials $g$ with simple fixed points and given multipliers at each fixed point is equivalent to determining differentials $\omega$ with simple poles and given residues at each pole. For this special case of simple poles, i.e., $b_1 = \cdots = b_n = 1$, Sugiyama described an analogous recursive solution, while leaving the general case of higher multiplicities as a conjecture (see \cite[Main Theorem III and Conjecture 2]{Su}). Our result thus settles the problem of counting polynomials with fixed multipliers completely for all cases.  

Finally we look into some future directions. It would be interesting to find out a generating function of those intersection numbers appearing in our computation 
 in the spirit of Witten's conjecture (see \cite{BR, BRZ} for related discussions). Moreover, one can consider similar questions for meromorphic differentials with more zeros, or in higher genera, or $k$-differentials.  Note that in general by dimension count fixing the residues at each pole may yield a positive-dimensional locus of differentials satisfying the given residues. Nevertheless, one can impose other non-residue conditions, such as fixing the moduli of the underlying curves or fixing certain configurations of saddle connections (in the sense of \cite{EMZ}), which can help reduce to finitely many solutions. Additionally, residues are part of local period coordinates for moduli spaces of meromorphic differentials with prescribed zero and pole orders. Therefore, our work can shed light on generalizing Masur--Veech volumes of moduli spaces of holomorphic differentials to the case of meromorphic differentials and consequently computing those volumes via intersection theory (see \cite{CMSZ} for the holomorphic case). We plan to treat these questions in future work. 

This paper is organized as follows. In Section~\ref{sec:review} we review the compactification of moduli spaces of differentials and introduce various divisor classes. In Section~\ref{sec:no-psv} we explain how to identify the universal line bundle class with the divisor class of the locus of differentials satisfying a general given residue tuple and prove  Theorem \ref{thm:intro} (i). In Section~\ref{sec:one-psv} we impose exactly one independent partial sum vanishing condition to the residues and prove Theorem \ref{thm:intro} (ii). Finally in Section~\ref{sec:any-psv} we prove Theorem~\ref{thm:main} for arbitrary residues and investigate combinatorial properties of the solution formula. We have also verified our formula numerically for a number of cases by using the software package \cite{diffstrata}. 

\subsection*{Acknowledgements} We thank Quentin Gendron, C\'esar Lozano Huerta, Myeongjae Lee, Johannes Schmitt, and Guillaume Tahar
for inspiring discussions. We are grateful to Laura DeMarco for pointing out the reference \cite{Su}. 

\section{Divisor classes}
\label{sec:review}

We use $\mathcal{P}_n(\mu)$ to denote the incidence variety compactification of the projectivized moduli space of differentials of type $\mu$ (see \cite{BCGGM1} for more details). 
The space $\mathcal{P}_n(\mu)$  parameterizes tuples $(X, [\omega], z,p_1,\ldots,p_n)$ where $[\omega]$ (up to scalar) is a stable differential of type $\mu$ compatible with a level graph of the Deligne--Mumford pointed stable rational curve $(X, z, p_1,\ldots,p_n)$. To simplify the notation we will also denote by $(X,[\omega])$ or just by $[\omega]$ a point in $\mathcal{P}_n(\mu)$. 

We introduce several divisor classes of $\mathcal{P}_n(\mu)$ that will be used later. Let $\psi_p$ be the psi-class associated to a marked point $p$. It is the divisor class of the bundle of the cotangent lines at $p$. In the case of genus zero the intersection of these psi-classes satisfies the formula
\begin{equation}
\label{eqn:psi}
    \int_{\overline{M}_{0,n}} \psi_{p_1}^{k_1} \cdots \psi_{p_n}^{k_n} = \binom{n-3}{k_1,\ldots, k_n}
\end{equation}
whenever $k_1 + \cdots + k_n=n-3$, as shown by Witten in \cite{W}. Next we define the universal line bundle class 
\begin{equation*}
 \xi =c_1(\mathcal{O}_{\mathcal{P}_n(\mu)}(1))
\end{equation*}
which is the dual of the tautological bundle $\mathcal{O}_{\mathcal{P}_n(\mu)}(-1)$ whose fiber over $[\omega]$ is spanned by $\omega$. We also need the boundary divisors $\delta_{\Gamma}$ whose general points parameterize stable pointed differentials compatible with a two-level graph $\Gamma$. Denote by $\Gamma_{\bot}$ the bottom level of $\Gamma$. Since every bottom component must contain a zero and 
 in our case there is a unique zero, the bottom level $\Gamma_{\bot}$ thus consists of a single component $X_0$ containing the marked zero $z$. Each irreducible top level component $X_i$ intersects $X_0$ at a single node $q_i$. Denote by $a_i={\rm ord}_{q_i}\omega|_{X_i}$ the zero order at 
 $q_i$ when restricting $\omega$ to $X_i$, denote by $t_i = a_i+1$ the twist of $\omega$ at $q_i$ (also called prong in \cite{BCGGM3}), and define 
 \begin{equation}
 \label{eqn:twist}
 t_{\Gamma} = \prod_{i=1}^m t_i
 \end{equation}
 to be the total twist of the graph $\Gamma$. These divisor classes satisfy the relation
\begin{equation}
\label{eqn:etamain}
    \xi = (b_i-1)\psi_{p_i} + \sum_{\Gamma_{\bot} \ni p_i} t_{\Gamma} \delta_{\Gamma}
\end{equation}
for each $i$ (see e.g., \cite[Theorem 6]{Sa}) where the sum runs over all two-level graphs whose bottom components contain the specified marked pole $p_i$. Note that (before twist) $\omega$ is identically zero on the bottom component $X_0$, and hence ${\rm Res}_{p_j}\omega=0$ for all $p_j\in X_0$.

We remark that there is another compactification finer than the incidence variety compactification, which is the multi-scale moduli space constructed in \cite{BCGGM3}. However, all divisor class relations and intersection calculations we will use hold identically on the multi-scale compactification. Therefore, we choose to work with the incidence variety compactification for the sake of simplifying the exposition.  

\section{General residues}
\label{sec:no-psv}

In this section we will compute the number $N(\mu,\rho)$ for the case (i) of Theorem \ref{thm:intro} when the given residues in $\rho=(r_1,\ldots,r_n)$ are general, i.e., when $\rho$ does not satisfy any partial sum vanishings. 

\begin{proof}[Proof of Theorem~\ref{thm:intro} {\rm (i)}]
Consider the following section of $\mathcal{O}_{\mathcal{P}_n(\mu)}(1)$ given by 
\begin{align*}
    \phi_i \colon\OO_{\mathcal{P}_{n}(\mu)}(-1) & \to  \CC\\
      (C, \omega, z, p_1,\ldots,p_n) & \mapsto r_n {\rm Res}_{p_i}\omega - r_i {\rm Res}_{p_n}\omega
\end{align*}
for $1\leq i \leq n-1$. The vanishing locus of $\phi_i$ is a divisor $D_i$ that parameterizes tuples $(X,[\omega],p_1,\ldots, p_n)$ such that 
$[{\rm Res}_{p_i}\omega : {\rm Res}_{p_n}\omega]=[r_i : r_n]$. Since $r_1+\cdots+r_n=0$, the intersection $\cap_{i=1}^{n-2} D_i$ in the interior of $\mathcal{P}_n(\mu)$ parameterizes those differentials $\omega$ whose residues are a scalar multiple of $\rho$. This intersection in the interior is transversal as the linear forms defining the sections $\phi_i$ are independent and the residues (modulo the total sum) are local coordinates of the interior of the moduli space. Moreover, no boundary point $(X, [\omega])$ lies in the intersection $\cap_{i=1}^{n-2} D_i$. Otherwise if $X$ is nodal, the poles contained in each irreducible component of $X$ lead to a partial sum vanishing by the Residue Theorem, contradicting the generality of $\rho$. Therefore, the intersection $\cap_{i=1}^{n-2} D_i$ consists of reduced points contained in the interior of $\mathcal{P}_n(\mu)$. Since the divisor class of $D_i$ is $c_1(\OO_{\mathcal{P}_n(\mu)}(1))=\xi$, we thus conclude that
\begin{equation*}
    N(\mu,\rho)=\int_{\mathcal{P}_n(\mu)} D_1\cdots D_{n-2}=\int_{\mathcal{P}_n(\mu)} \xi^{n-2}. 
\end{equation*}
It remains to show that the above intersection number is equal to $f(a,n)=\frac{a!}{(a-(n-2))!}$ as defined in \eqref{eqn:s}. 

We proceed by induction on $n$. For $n=2$, the space $\mathcal{P}_2(\mu)$ is a single point and our claim holds obviously. 
For $n\geq 3$, by using relation \eqref{eqn:etamain} with respect to the pole $p_1$, we have 
\begin{equation*}
    \xi^{n-2}=(b_1-1)\xi^{n-3}\psi_{p_1}+ \sum_{\Gamma_{\bot} \ni p_1}t_{\Gamma} \delta_{\Gamma}\xi^{n-3}.
\end{equation*}
Whenever we have a product $\xi^{n-2-k}\psi_{p_1}^k$, we can continue expanding a factor $\xi$ by using \eqref{eqn:etamain} with respect to the pole $p_1$. Therefore, we end up with the expression
\begin{equation}
\label{eqn:etaexpand}
     \xi^{n-2} =  \sum_{k=0}^{n-2} (b_1-1)^k \sum_{\Gamma_{\bot} \ni p_1}t_{\Gamma} \delta_{\Gamma}\xi^{n-3-k}\psi_{p_1}^k.
\end{equation}

We now compute the products $\delta_{\Gamma}\xi^{n-3-k}\psi_{p_1}^k$. Suppose the two-level graph $\Gamma$ has $m$ top components $X_1,\ldots,X_m$ and one bottom component $X_0$ where each component contains $n_i$ marked poles for $i=0,\ldots,m$. Then the corresponding twisted differential on $X_0$ has a single zero at $z$, $n_0$ marked poles including $p_1$, and $m$ additional poles at the nodes $q_i = X_i \cap X_0$ for $i=1,\ldots, m$. Denote by $\mu_i$ the tuple of orders of the zero (at the node $q_i$) and marked poles of $\omega|_{X_i}$ for $i=1,\ldots, m$. Denote the tautological bundle of $\mathcal{P}_{n_i}(\mu_i)$ by $\OO_i(-1)$. Consider the product space
\begin{equation*}
    \mathcal{P} = \mathcal{P}_{n_1}(\mu_1)\times \cdots \times \mathcal{P}_{n_m}(\mu_m).
\end{equation*}
There exist the projection maps $\pi_i\colon \mathcal{P} \to \mathcal{P}_{n_i}(\mu_i)$ and the vector bundle 
\begin{equation}
\label{eqn:E}
\mathcal{E}=\pi_1^{*}\OO_1(-1)\oplus \cdots \oplus \pi_m^{*}\OO_m(-1) 
\end{equation}
on $\mathcal P$. Note that $\delta_{\Gamma}$ can be identified with $\PP\mathcal{E}\times \M_{0,n_0+m+1}$ where $\PP\mathcal{E}$ parameterizes differentials on the top components and $\M_{0,n_0+m+1}$ parameterizes the marked points and nodes in the bottom component. In particular, $\psi_{p_1}$ is trivial on the top moduli space $\PP\mathcal{E}$ and $\xi|_{\delta_\Gamma}$ as the pullback of $c_1(\OO_{\PP\mathcal{E}}(1))$ is trivial on the bottom moduli space $\M_{0,n_0+m+1}$. Therefore, $\delta_{\Gamma}\xi^{n-3-k}\psi_{p_1}^k$ is nonzero only if 
$$k = \dim \M_{0,n_0+m+1} = n_0 + m -2. $$

Denote by $\xi_i =c_1(\OO_i(1))$ for $i=1,\ldots,m$ and for simplicity we still use $\xi_i$ for the pullback class via $\pi_i$. The Segre class of $\mathcal{E}$ can be computed as 
\begin{align}
\label{eqn:segre}
    s(\mathcal{E})  &=\frac{1}{\prod_{i=1}^m (1-\xi_i)} \\ \nonumber
                  &=\prod_{i=1}^m (1+\xi_i+\cdots+\xi_i^{d_i})
\end{align}
where $d_i = n_i - 2 = \dim \mathcal{P}_{n_i}(\mu_i)$. Therefore, we obtain that 
\begin{align*}
    \int_{\mathcal{P}_{n}(\mu)} \delta_{\Gamma}\xi^{n-n_0-m-1}\psi_{p_1}^{n_0+m-2} 
    &=\int_{\PP\mathcal{E}}\xi^{n-n_0-m-1} \int_{\M_{0,n_0+m+1}} \psi_{p_1}^{n_0+m-2} \\
    &=\int_{\PP\mathcal{E}}\xi^{n-n_0-m-1} \quad {\rm by}\ \eqref{eqn:psi} \\
    &=\int_{\mathcal{P}}s_{n-n_0-2m}(\mathcal{E}) \\
    &=\int_{\mathcal{P}} \xi_1^{d_1}\cdots \xi_m^{d_m} \\
    &=\prod_{i=1}^m \int_{\mathcal{P}_{n_i}(\mu_i)} \xi_i^{d_i} \\
    &=\prod_{i=1}^m f(a_i,n_i) 
\end{align*}
where $a_i={\rm ord}_{q_i}\omega|_{X_i}$ is the zero order at the node $q_i=X_i\cap X_0$ for $i=1,\ldots,m$ and the last equality follows from induction. 

Next recall that 
$$t_{\Gamma}=\prod_{i=1}^m (a_i+1).$$ 
Hence we have the equality 
\begin{align}
    t_{\Gamma}\int_{\mathcal{P}_{n}(\mu)} \delta_{\Gamma}\xi^{n-n_0-m-1}\psi_{p_1}^{n_0 + m -2}&= \prod_{i=1}^m  (a_i+1) f(a_i,n_i) \nonumber \\
    &=\prod_{i=1}^m f(a_i+1,n_i+1)\label{eqn:deltaetapsi}
\end{align}   
which is a polynomial of degree 
$$\sum_{i=1}^m (n_i-1)=n-n_0-m = n-k-2$$ 
in $\mathbb{Q}[b_1,\ldots,b_n]$ by using the relation $a_i=\big(\sum_{p_j\in X_i}b_j \big) -2$ and $k = n_0+m-2$. 

Consider the polynomial
\begin{align*}
  P(b_1,\ldots,b_n) &= f(a,n)- \int_{\mathcal{P}_n(\mu)} \xi^{n-2} \\
    &=f(a,n) - \Bigg(\sum_{k=0}^{n-2}(b_1-1)^k \sum_{\Gamma_{\bot} \ni p_1}t_{\Gamma} \int_{\mathcal{P}_n(\mu)} \delta_{\Gamma}\xi^{n-3-k}\psi_{p_1}^k \Bigg) \quad {\rm by} \ \eqref{eqn:etaexpand}.
\end{align*}
It suffices to show that $P$ is identically zero. First observe that $\deg P\leq n-2$. Moreover for $b_1=1$, we have 
\begin{align*}
        P(1,b_2,\cdots,b_n)&=f(a,n)-\Bigg( \sum_{\Gamma_{\bot} \ni p_1}t_{\Gamma} \int_{\mathcal{P}_n(\mu)}\delta_{\Gamma}\xi^{n-3}\Bigg) \\
        &=f(a,n)-f(a_1+1,n_1+1) \\
        &=0 
\end{align*}
where the last two equalities follow from \eqref{eqn:deltaetapsi} for the case $k=n_0+m-2=0$ and hence $n_0=1$, $m=1$, $n_1=n-1$, and $a_1+1=\big(\sum_{i=2}^n b_i \big)-1=a$. We  thus conclude that $b_1-1$ divides $P(b_1,\ldots, b_n).$ By symmetry $\prod_{i=1}^n (b_i-1)$ divides $P(b_1,\ldots,b_n)$. Since $\deg P \leq n-2$, we must have $P \equiv 0$, and hence $ \int_{\mathcal{P}_{n}(\mu)}\xi^{n-2} = f(a,n).$
\end{proof}
 
The results and ideas in the above proof will also help us treat other previously unknown cases of special residue conditions.

\section{One partial sum vanishing}
\label{sec:one-psv}

In this section we study the case of a given residue tuple $\rho = (r_1,\ldots, r_n)$ with a unique pair of complementary proper 
 subsets $I\sqcup I^c = \{1, \ldots, n \}$ such that $\sum_{j\in I} r_j=0$ as in Theorem \ref{thm:intro} (ii). Although this is a special case of our most general formula \eqref{eqn:main}, we demonstrate it separately to motivate the strategy.  

\begin{proof}[Proof of Theorem~\ref{thm:intro} {\rm (ii)}]
Consider the following section of $\OO_{\mathcal{P}_{n}(\mu)}(1)$ given by 
\begin{align*}
    \phi_I\colon  \OO_{\mathcal{P}_{n}(\mu)}(-1) & \to  \CC\\
      (C,\omega,z,p_1,\ldots,p_n) &  \mapsto \sum_{i\in I} {\rm Res}_{p_i}\omega. 
\end{align*}
The vanishing locus of $\phi_I$ in $\mathcal{P}_{n}(\mu)$ has divisor class $\xi$ and contains those boundary divisors $\delta_{\Gamma}$ with multiplicity $t_{\Gamma}$ such that every top component has marked poles either all in $I$ or all in $I^c$. We denote this set of boundary divisors by $G_I$. Let $\cR_I$ be the closure of the interior vanishing locus. Then the divisor class of $\cR_I$ can be expressed as 
\begin{equation}
\label{eqn:RI}
    \cR_I =\xi-\sum_{\Gamma \in G_I} t_{\Gamma} \delta_{\Gamma} 
\end{equation}
(see \cite[Proposition 8.3]{CMZ}). 

First consider the case $2 \leq |I| \leq n-2$.  Assume without loss of generality that $n-2\in I$. Take the divisors $D_1,\ldots,D_{n-3}$ as defined in Section~\ref{sec:no-psv}. If $D_1\cap\cdots \cap D_{n-3}\cap \cR_I$ contains a boundary point, then the poles indexed by $I$ and $I^c$ must be the marked poles contained in two top level components $X_1$ and $X_2$, respectively, for otherwise $\rho$ would satisfy an extra partial sum vanishing. However, the twisted differential on the bottom component $X_0$ with the single zero $z$ would have zero residues at the nodes $q_i = X_i \cap X_0$ for $i = 1,2$ by using the global residue condition in \cite{BCGGM1}, but such a differential does not exist by \cite[Lemma 3.6]{BCGGM1}. We thus conclude that $D_1\cap\cdots \cap D_{n-3}\cap \cR_I$ consists of interior points only, which implies that 
\begin{align*}
   N(\mu,\rho)&=\int_{\mathcal{P}_{n}(\mu)} D_1\cdots D_{n-3}\cR_I \\
              &=\int_{\mathcal{P}_{n}(\mu)} \xi^{n-3}\left(\xi-\sum_{\Gamma \in G_I} t_{\Gamma}\delta_{\Gamma}\right) \quad {\rm by}\ \eqref{eqn:RI}\\
              &=\int_{\mathcal{P}_{n}(\mu)} \xi^{n-2}- \sum_{\Gamma \in G_I} t_{\Gamma}\int_{\mathcal{P}_{n}(\mu)}\xi^{n-3}\delta_{\Gamma}. 
\end{align*}
In Section~\ref{sec:no-psv} we showed that $\int_{\mathcal{P}_{n}(\mu)} \xi^{n-2}=f(a,n)$ and that $\xi^{n-3}\delta_{\Gamma}$ is nonzero only if $m+n_0-2=k=0$, that is, when $m=2$ and $n_0=0$. In this case the bottom component $X_0$ of $\Gamma$ contains $z$ and the nodes $q_1, q_2$.  
By using \eqref{eqn:deltaetapsi} we have 
\begin{equation}
    t_{\Gamma}\int_{\mathcal{P}_{n}(\mu)} \delta_{\Gamma}\xi^{n-3}=f(a_1+1,|I|+1) \cdot f(a_2+1,n-|I|+1)
\end{equation}
where $a_1=b_I-2$ and $a_2=b_{I^c}-2=a- b_I$. It follows that
\begin{align*}
    N(\mu,\rho) &=f(a,n)-f(b_I-1,|I|+1)\cdot f(a- b_I +1,n-|I|+1). 
\end{align*}

The remaining case when $|I| = 1$ or $|I^c| = 1$ can be treated similarly. Suppose $I = \{ 1, \ldots, n-1\}$ and then the boundary correction term in the above is from the graph $\Gamma$ where the top component $X_1$ contains $p_1,\ldots, p_{n-1}$, the bottom component $X_0$ contains $z, p_n$, the node $q_1 = X_1 \cap X_0$, and the twist $t_\Gamma = a - b_n + 1$. Since $p_n$ is in the bottom, the stable differential $\omega$ (before twist) automatically has zero residue at $p_n$ as $\omega$ is identically zero on $X_0$. 
Note that $f(b_n - 1, 2) = 1$ and the bottom component of $\Gamma$ has unique moduli. In this case we conclude that 
 \begin{align*}
  N(\mu,\rho) &=  f(a,n) - t_\Gamma f(a-b_n, n-1) \\
   &=  f(a, n) - f(b_n-1, 2) \cdot f(a-b_n+1, n)
  \end{align*}
  as desired. 
  \end{proof}
  
  \begin{remark}
  \label{rem:semistable}
  In the last case when $I = \{ 1,\ldots, n-1\}$ and $I^c = \{n\}$, to make a uniform expression as before, i.e., to have two top components containing the marked poles indexed by $I$ and $I^c$ respectively for the boundary correction term, we can bubble up $p_n$ from the bottom component $X_0$ of $\Gamma$ to form a semistable top component $X_2$. Denote by $\Gamma'$ the resulting graph that has an additional node $q_2 = X_2 \cap X_0$ playing the role of $p_1$ in $X_0$. Then $t_{\Gamma'} = (b_n-1) t_\Gamma$ where $b_n-1$ is the twist at the new node $q_2$. Note that by definition 
  $f(b_n-2, 1) = 1/(b_n-1)$. Therefore, in this case we have 
      \begin{align*}
 t_{\Gamma'} f(b_n-2, 1)\cdot f(a-b_n, n-1) = t_\Gamma f(a-b_n, n-1). 
  \end{align*}
  In other words, even though the factor 
$ f(b_n-2, 1)$ counts the semistable component and is fractional, it cancels out with the twisting factor gained from the semistable node and hence does not affect the answer. The advantage of this semistable version is that the expression of the boundary correction term remains symmetric with respect to $I\sqcup I^c = \{1,\ldots, n\}$ and all marked poles can be regarded as in top level components. In what follows we will adapt this viewpoint to simplify the combinatorial description and notation in our proof. We remark that the bottom component $X_0$ remains stable, as any marked pole bubbled up is replaced by a node in $X_0$ and hence the total number of special points in $X_0$ does not change. 
\end{remark}

\section{Multiple partial sum vanishings}
\label{sec:any-psv}

In this section we study the case of an arbitrary residue tuple $\rho = (r_1,\ldots, r_n)$ and prove the formula~\eqref{eqn:main} in Theorem \ref{thm:main}. Since the Residue Theorem requires the total sum of residues to be zero, we assume that the given residue tuple $\rho$ always satisfies the total sum vanishing $\sum_{i=1}^n r_i = 0$ as a prerequisite. 

First we introduce some terminologies that can characterize the partial sum vanishings satisfied by $\rho$. For a subset $I$ of $\{1,\ldots, n\}$, denote by $r_I = \sum_{j\in I} r_j$ the linear form recording the sum of assigned residues indexed by $I$. Let $\mathcal{I}=\{I_i\}_{i=1}^k$ be a collection of nonempty proper subsets such that each partial sum vanishing $r_{I_i}$ for $I_i\in \cI$ is satisfied by $\rho$. We say that $\mathcal{I}$ is complete with respect to $\rho$ if every partial sum vanishing $r_J = 0$ satisfied by $\rho$ can be generated by those in $\mathcal I$ together with the total sum $r_1 + \cdots + r_n$, i.e., $r_J$ can be generated by those $r_{I_i}$ for $I_i \in \mathcal I$ modulo the total sum. We say that $\mathcal{I}$ is independent if $r_{I_1}, \ldots, r_{I_k}$ are independent linear forms modulo the total sum.

From now on we take a complete and independent collection $\mathcal{I}=\{I_1,\ldots,I_k\}$ with respect to the given residue tuple $\rho$. Consider the sequence $$\emptyset = \mathcal{I}_0 \subset \mathcal{I}_1\subset \cdots \subset \mathcal{I}_k=\mathcal{I}$$
where $\mathcal{I}_i=\{I_1,\ldots,I_i\}$. Define the subspace $\cR_i\subset \mathcal{P}_n(\mu)$ as the closure of the locus of differentials in the interior of the moduli space that satisfy the partial sum vanishings generated by $\mathcal{I}_i$. Then we have 
$$ \mathcal{P}_n(\mu) = \cR_0 \supset \cR_1 \supset \cdots  \supset \cR_k $$
where each $\cR_j$ has codimension $j$ in $\mathcal{P}_n(\mu)$. As the $k$ partial sum vanishings in $\mathcal{I}$ are independent, the hyperplanes defined by them intersect transversally. Therefore, there exist additionally $n-2-k$ general and transversal hyperplanes, together with $r_{I_i}$ for $i=1, \ldots, k$, that cut out the given point $\rho = [r_1:\cdots:r_n]$ as a projective tuple (and see Lemma~\ref{lem:zero} for the exceptional case when $\rho$ is identically zero). Let $D_1,\ldots,D_{n-2-k}$ be the divisors in $\mathcal{P}_n(\mu)$ satisfying these additional general residue hyperplanes respectively where the class of $D_i$ is $\xi$ as seen before.  

We claim that the intersection $D_1\cap \cdots \cap D_{n-2-k} \cap \cR_k$ consists of interior points of $\mathcal{P}_n(\mu)$ only. Otherwise, suppose a boundary point $(X, [\omega])$ compatible with a level graph $\Gamma$ lies in this intersection locus. To simplify the description and make the notation more symmetric, as explained in Remark~\ref{rem:semistable}
 we use a semistable model of the underlying curve $X$ by bubbling up any marked pole $p_i$ in a lower level component of $\Gamma$ to form a new top level semistable component that contains $p_i$ and meets the original component at a node, where $\omega$ restricted to the new component has a single pole at $p_i$ and a single zero at the node. Since the residue of $\omega$ at $p_i$ remains zero before and after, this change of viewpoints does not affect any argument of applying the Residue Theorem or the global residue condition to our situation. Now with the $n$ marked poles all on top level, suppose $\Gamma$ has $m$ top components $X_1,\ldots, X_m$ and denote by $I_{X_i}$ the set of indices of the marked poles contained in each $X_i$. Then the sum of the residues of the marked poles in each $X_i$ is zero, which defines a partial sum vanishing indexed by $I_{X_i}$ and satisfied by $\rho$. Then by the (generalized) global residue condition on $\cR_k$ (see \cite[Section 4.1]{CMZ}), the residue of the twisted differential associated to $\omega$ at every node of the bottom component of $\Gamma$ must be zero. However since the bottom twisted differential has a single zero at $z$, 
 its residues cannot be all zero (see \cite[Lemma 3.6]{BCGGM1}), which leads to a contradiction. Therefore, we conclude that 
 
 \begin{equation}
 \label{eqn:N-any}
    N(\mu,\rho)=\int_{\mathcal{P}_n(\mu)} \cR_k D_1\cdots D_{n-2-k}=\int_{\cR_k}\xi^{n-2-k}.
\end{equation}

Now to prove Theorem~\ref{thm:main} it suffices to prove the following result. 

\begin{prop}
\label{prop:any}
In the above setting we have 
\begin{equation} 
\label{eqn:generalformula}
    \int_{\cR_k}\xi^{n-2-k} = \sum_{s=1}^S (-1)^{s-1}(a+1)^{s-2}\sum_{\mathcal{J}_s\in \mathcal{C}_s} \prod_{J_j\in \mathcal{J}_s} f_{J_j}
\end{equation}
where $\mathcal{C}_s$ parameterizes collections of $s$ disjoint partial sum vanishings of $\rho$ whose union contains all the $n$ poles, $S$ is the maximum value of such $s$, and $f_{J_j}=f(b_{J_j}-1,|J_j|+1)$.
\end{prop}

Recall that if we require every residue to be zero, i.e., if the given residue tuple $\rho$ is $\underline{0} = (0,\ldots,0)$, such a differential does not exist, and hence $N(\mu, \underline{0}) = 0$. As a preparation step we first prove that the right-hand side of \eqref{eqn:generalformula} is zero for the case of $\rho = \underline{0}$, which also yields an interesting combinatorial identity.  

\begin{lem}
\label{lem:zero} Denote by 
$$P(\mu, \underline{0}) = \sum_{s=1}^n (-1)^{s-1}(a+1)^{s-2}\sum_{J_1\sqcup \cdots \sqcup J_s = \{1,\ldots, n\}} \prod_{j=1}^s f_{J_j}.$$
Then $P(\mu, \underline{0})$ is identically zero.  
\end{lem}

\begin{proof} To make the notation clear, we will regard each $J_i$ as a subset of $\{ p_1, \ldots, p_n\}$ consisting of the poles indexed by $J_i$. 
Recall that $a + 1 =b_1+\cdots+b_n-1$ and 
$$f_{J_j}=f(b_{J_j}-1,|J_i|+1)=\frac{(b_{J_j}-1)!}{(b_{J_j}-|J_i|)!}$$ 
where $b_{J_j}=\sum_{p_i\in J_j}b_i$. Treat $b_1, \ldots, b_n$ as variables. Then each term $(-1)^{s-1}(a+1)^{s-2}\prod_{j=1}^s f_{J_j}$ is a polynomial in $b_1, \ldots, b_n$ of degree $n-2$ for every partition of $\{1,\ldots,n\}$ into $s$ disjoint subsets $J_1, \ldots, J_s$. Therefore, $P(\mu, \underline{0})$ as a polynomial in $b_1, \ldots, b_n$ has degree at most $n-2$. 

Set $b_1=1$ and consider a partition $J_1,\ldots,J_s$ such that $J_1=\{p_1\}$. Then
\begin{align*}
  &\quad  (-1)^{s-1}(b_1+\cdots+b_n-1)^{s-2} \prod_{j=1}^s f_{J_j}  \Big| _{b_1=1} \\ 
  &= (-1)^{s-1}(b_{J_2}+\cdots +b_{J_s})^{s-2}\prod_{j=2}^s f_{J_j} \\
 & =  -(-1)^{s-2}(b_{J_2}+\cdots +b_{J_s})^{s-3}\sum_{i=2}^s \left(b_{J_i}\prod_{j=2}^s f_{J_j}\right) \\
 & = -\sum_{i=2}^s\left((-1)^{s-2}(b_1+b_{J_2}+\cdots +b_{J_s}-1)^{s-3} f(b_{J_i}+b_1-1,|J_i|+2)\right)\Big|_{b_1=1}\prod_{j\not= 1,i}f_{J_j}\\
 & = -\sum_{i=2}^s\left((-1)^{s-2}(b_1+b_{J_2}+\cdots +b_{J_s}-1)^{s-3} f(b_{J'_i}-1,|J'_i|+1)\right)\Big|_{b_1=1}\prod_{j\not= 1,i}f_{J_j}
\end{align*}
where $J'_i=J_i\cup \{p_1\}$ for $i = 2, \ldots, r$. Therefore, the terms that have $\{p_1\}$ as a part of the partition in the sum cancel out with the remaining terms. It implies that $P(\mu,\underline{0})|_{b_1=1}=0$ and hence $b_1-1$ divides $P(\mu,\underline{0})$. By symmetry we conclude that $\prod_{i=1}^n (b_i-1)$ divides $P(\mu,\underline{0})$. 
Since the degree of $P(\mu,\underline{0})$ is at most $n-2$, it follows that $P(\mu,\underline{0})\equiv 0$. 
\end{proof}

\begin{proof}[Proof of Proposition~\ref{prop:any}]
By the above lemma we only need to consider the case $\rho\neq (0, \ldots, 0)$. We will prove \eqref{eqn:generalformula} by induction on $k$. For $k=0$, recall that $\cR_0=\mathcal{P}_n(\mu)$ and 
\begin{align*}
    \int_{\mathcal{P}_n(\mu)}\xi^{n-2} &= f(a,n) =(a+1)^{-1}f(a+1,n+1)
\end{align*}
as shown in Section~\ref{sec:no-psv}. 

Suppose the claim holds for imposing up to $ k-1$ independent partial sum vanishing conditions. To obtain the divisor class of $\cR_k$ inside $\cR_{k-1}$, we consider the following section of the line bundle $\mathcal{O}_{\cR_{k-1}}(1)$ given by 
\begin{align*}
    \phi_k \colon  \cR_{k-1} & \to  \mathbb{C}\\
      (C,\omega,z,p_1,\ldots,p_n) &  \mapsto  \sum_{p_j \in I_k}\text{Res}_{p_j}\omega.
\end{align*}
The vanishing locus of $\phi_k$ has divisor class $\xi$. Moreover, this locus contains boundary points in $\delta_{\Gamma}$ for some two-level graphs $\Gamma$ such that the partial sum vanishings in $\mathcal{I}_k$ are generated by the partial sum vanishings in $\mathcal{I}_{k-1}$ together with those partial sum vanishings $I_{X_i}$ generated by the Residue Theorem applied to each top component $X_i$ of $\Gamma$. Denote the set of such two-level graphs by $G_k$. Here for simplifying the combinatorial description we again use the semistable model of $\Gamma$ by bubbling up each marked pole in lower level to a semistable top level component, which as we explained before does not affect relevant residue conditions 
and calculations.  We then obtain the following relation of divisor classes 
\begin{equation*}
    \cR_k = \xi-\sum_{\Gamma\in G_k}t_{\Gamma}\delta_{\Gamma}
\end{equation*}
in $\cR_{k-1}$, where $t_{\Gamma}$ is the total twist of $\Gamma$ as in~\eqref{eqn:twist}.  
It follows that 
\begin{equation}
\label{eqn:minusboundary}
    \int_{\cR_k}\xi^{n-2-k} =\int_{\cR_{k-1}}\xi^{n-1-k} - \sum_{\Gamma \in G_k}t_{\Gamma}\int_{\cR_{k-1}}\delta_{\Gamma}\xi^{n-2-k}.
\end{equation} 

We can compute the first term on the right-hand side of \eqref{eqn:minusboundary} by using the induction hypothesis. For the remaining terms, take a two-level graph $\Gamma \in G_k$ where $\Gamma$ has $m$ top components $X_1,\ldots, X_m$ and a bottom component $X_0$ with the zero $z\in X_0$ and the nodes $q_i=X_i\cap X_0$ for $i=1,\ldots,m$.
Recall that $\delta_\Gamma$ can be identified with $\PP\cE\times \M_{0,m+1}$ (see \eqref{eqn:E} and here the bottom component $X_0$ contains only $z, q_1,\ldots, q_m$ as all lower level marked poles are replaced by semistable top components). Moreover, recall that $\xi$ is trivial restricted to the bottom moduli space. Therefore, any nonzero intersection number $\delta_\Gamma \xi^{n-2-k}$ in $\cR_{k-1}$ consists of two parts contributed from the top and the bottom respectively. On the top level we count the number of differentials on $X_1, \ldots, X_m$ that satisfy the given partial sum vanishing conditions that remain on top as well as the other $n-2-k$ general residue conditions imposed from $\xi^{n-2-k}$. On the bottom level we count the number of twisted differentials on $X_0$ with residue conditions at the nodes $q_i$ imposed by the (generalized) global residue condition of $\cR_{k-1}$. 

First consider the bottom level. Denote by $\omega_0$ the bottom twisted differential of zero and pole type $\mu_0$. Suppose $\rho_0$ is a residue tuple of $\omega_0$ that satisfies the global residue condition induced by the partial sum vanishings in $\cI_{k-1}$ that define $\cR_{k-1}$. Note that the partial sum vanishings in $\mathcal{I}_k$ are generated by those in $\mathcal{I}_{k-1}$ and $I_{X_1},\ldots, I_{X_m}$ from the Residue Theorem on each top component. In order to have $\delta_\Gamma \xi^{n-2-k}$ nonzero in $\cR_{k-1}$, the bottom moduli space has to be zero-dimensional. Moreover, $\rho_0\neq \underline{0}$ as at least one $I_{X_i}$ is not generated by those in $\mathcal{I}_{k-1}$. Therefore, the global residue condition imposes $m-2$ independent partial sum vanishings to the residues of $\omega_0$ at the nodes $q_1, \ldots, q_m$ which determine $\rho_0$ modulo scaling.  
By using the induction hypothesis, the bottom level contribution with the residue tuple $\rho_0$ to the intersection number $\delta_\Gamma \xi^{n-2-k}$ in $\cR_{k-1}$ is 
\begin{equation}
\label{eqn:N_0}
N(\mu_0,\rho_0)=\sum_{s=1}^{S_0} (-1)^{s-1}(a+1)^{s-2}\sum_{\mathcal{J}^0_s\in \mathcal{C}^0_{s}} \prod_{J^0_j\in \mathcal{J}^0_s} f_{J^0_j} 
\end{equation}
where $\mathcal{C}^0_{s}$ parameterizes the collections $\mathcal{J}^0_s=\{J^0_1,\ldots,J^0_s\}$ of $s$ disjoint partial sum vanishings satisfied by $\rho_0$ whose union consists of the $m$ nodes in $X_0$, and $S_0$ is the maximum value of such $s$.

Next consider the top level. Suppose a partial sum vanishing condition $J$ generated by $\mathcal{I}_{k-1}=\{I_1,\ldots,I_{k-1}\}$ splits into different top components, i.e., suppose $J$ is not a union of some partial sum vanishings generated by $\mathcal{I}_{k-1}$ each of which is completely contained in a top component. Then imposing the residue condition $J$ is the same as intersecting with $\xi$, since there is no boundary correction term that would imply $J$ by applying the Residue Theorem to a single component. Therefore, we can separate the partial sum vanishings generated by $\mathcal{I}_{k-1}$ into two parts, those that are contained in one top component and those that split into different top components (and hence can be substituted by powers of $\xi$). 

On each top component $X_i$ for $i=1,\ldots, m$, let $n_i$ be the number of marked poles and $\mu_i$ the tuple of zero and pole orders. Denote by 
$\cR^i_{k_i}\subset \mathcal{P}_{n_i}(\mu_i)$ the closure of the locus of differentials satisfying the partial sum vanishings generated by $\mathcal{I}_{k-1}$ that are contained in $X_i$, where $k_i$ is the codimension of $\cR^i_{k_i}$ in $\cP_{n_i}(\mu_i)$. 
Denote by $\cR_{\Gamma}\subset \PP\cE$ the closure of the locus of top level differentials satisfying the partial sum vanishings generated by $\mathcal{I}_{k-1}$ that are completely contained in one top component of $\Gamma$. The codimension of $\cR_{\Gamma}$ in $\PP\cE$ is $k'=k_1+\cdots+k_m$. As $m-2$ independent residue conditions are imposed to the bottom level moduli space of $\Gamma$, there are $k-1-(m-2)$ independent residue conditions imposed to the top level moduli space $\mathbb{P}\mathcal{E}$. This implies that $k+1-m-k'$ independent partial sum vanishings generated by $\mathcal{I}_{k-1}$ can each be substituted by $\xi$ as explained above. Therefore, the top level contribution to the intersection number $\delta_\Gamma \xi^{n-2-k}$ in $\cR_{k-1}$ is 
\begin{equation*}
    \int_{\mathbb{P}\mathcal{E}} \cR_{\Gamma}\xi^{n-1-m-k'}. 
\end{equation*}

Recall that the Segre class $s(\mathcal{E})$ in \eqref{eqn:segre} is 
\begin{equation*}
    s(\mathcal{E}) =\prod_{i=1}^m(1+\xi_i+\xi_i^2+\cdots+\xi_i^{d_i})
\end{equation*}
where 
$d_i=\dim \cR^i_{k_i}=n_i-2-k_i$ for $i=1,\ldots,m$. Since 
    $\sum_{i=1}^m d_i = n-2m-k'$
and  $\text{rank} \ \mathbb{P}\mathcal{E}=m-1$ over $\cR^1_{k_1}\times\cdots \times \cR^m_{k_m}$, we conclude that
\begin{align*}
    \int_{\mathbb{P}\mathcal{E}} \cR_{\Gamma} \xi^{n-1-m-k'}&=s_{n-2m-k'}(\mathcal{E})\\
    &=\prod_{i=1}^m \int_{\cR^i_{k_i}}\xi_i^{d_i}\\
    &=\prod_{i=1}^m N(\mu_i,\rho_i)
\end{align*}
where $\rho_i$ is the tuple of residues at the marked poles in $X_i$ satisfying the partial sum vanishings that define $\cR^i_{k_i}$ (and general other than that) for $i=1,\ldots,m$. 

Summarizing the above discussion, we obtain that 
\begin{equation}
\label{eqn:product}
    \int_{\cR_{k-1}} \delta_{\Gamma}\xi^{n-2-k} = \prod_{i=0}^m N(\mu_i,\rho_i). 
\end{equation}
Now we can use the induction hypothesis to express every term in this product. Recall that $t_{\Gamma}=\prod_{i=1}^m (b_{X_i}-1)$ where $b_{X_i} = \sum_{p_j\in X_i} b_j$.  
Then
\begin{align*}
    t_{\Gamma}\int_{\cR_{k-1}} \delta_{\Gamma}\xi^{n-2-k} &= N(\mu_0,\rho_0) \cdot \prod_{i=1}^m \left(\sum_{s_i=1}^{S_i} (-1)^{s_i-1}(b_{X_i}-1)^{s_i-1}\sum_{\mathcal{J}^i_{s_i}\in \mathcal{C}^i_{s_i}} \prod_{J^i_j\in \mathcal{J}^i_{s_i}} f_{J^i_j}\right)
\end{align*}
where $\cC^i_{s_i}$ parameterizes collections $\mathcal{J}^i_{s_i}=\{J^i_1,\ldots,J^i_{s_i}\}$ of $s_i$ disjoint partial sum vanishings satisfied by $\rho_i$ whose union consists of all the poles in $X_i$, and $S_i$ is the maximum value of such $s_i$. Recall the expression of $N(\mu_0,\rho_0)$ in \eqref{eqn:N_0} where $a+1=\sum_{i=1}^m b_{X_i}-1$, $f_{J^0_j}=f(b_{J^0_j}-1,|J^0_j|+1)$, and $b_{J^0_j}=\sum_{q_i \in J^0_j}b_{X_i}$ is the sum of the pole orders in the top components over the nodes $q_i$ that belong to $J^0_j$. This implies that $N(\mu_0,\rho_0)$ is a polynomial of degree $m-2$ in the variables $b_{X_1},\ldots,b_{X_m}$. 

We say that a collection $\mathcal{J}_s$ of disjoint partial sum vanishings that consist of all the $n$ poles is compatible with a two-level graph $\Gamma$ if each element of $\mathcal{J}_s$ is completely contained in a top component of $\Gamma$. Then we have 
\begin{equation*}
    t_{\Gamma}\int_{\cR_{k-1}} \delta_{\Gamma}\xi^{n-2-k}  = \sum_{s=2}^{S_{\Gamma}} \sum_{\mathcal{J}_s\in \mathcal{C}^{\Gamma}_{s}}Q^{\Gamma}_{ \mathcal{J}_s}(b_{X_1},\ldots,b_{X_s})\prod_{J_j\in \mathcal{J}_s}f_{J_j}
\end{equation*}
where $S_{\Gamma}=\sum_{i=1}^m S_i$, the set $\mathcal{C}^{\Gamma}_{s}$ parameterizes collections $\mathcal{J}_s$ of $s$ disjoint  partial sum vanishings that are compatible with $\Gamma$, and
\begin{equation*}
    Q^{\Gamma}_{\mathcal{J}_s}(b_{X_1},\ldots,b_{X_m})=N(\mu_0,\rho_0)(-1)^{s-m}\prod_{i=1}^m(b_{X_i}-1)^{s_i-1}
\end{equation*}
where $s_i$ is the number of elements in $\mathcal{J}_s$ that are completely contained in $X_i$. Note that $s=\sum_{i=1}^m s_i$ and the polynomial $Q^{\Gamma}_{\mathcal{J}_s}$ has degree $s-2$ in the variables $b_{X_1},\ldots, b_{X_m}$.  

Let $\mathcal{J}_s=\{J_1,\ldots,J_s\}$, $b_{J_j}=\sum_{i\in J_j}b_i$, and then $b_{X_i}=\sum_{J_j\subset I_{X_i}}b_{J_j}$. We consider the polynomial
\begin{equation}
    P_{\mathcal{J}_s}^{\Gamma}(b_{J_1},\ldots, b_{J_s})=Q_{\mathcal{J}^s}^{\Gamma}(b_{X_1},\ldots,b_{X_m})
\end{equation}
by using the variables $b_{J_j}$ instead of $b_{X_i}$, where the degree of $P_{\mathcal{J}_s}^{\Gamma}$ is still $s-2$. Then we can rewrite
\begin{align*}
    t_{\Gamma}\int_{\cR_{k-1}} \delta_{\Gamma}\xi^{n-2-k}  =& \sum_{s=2}^{S_{\Gamma}} \sum_{\mathcal{J}_s\in \cC^{\Gamma}_s} P^{\Gamma}_{ \mathcal{J}_s}(b_{J_1},\ldots,b_{J_s})\prod_{J_j\in \mathcal{J}_s}f_{J_j}.
\end{align*}

By applying the induction hypothesis to \eqref{eqn:minusboundary} we obtain that
\begin{equation*}
    N(\mu,\rho)=\sum_{s=1}^S \sum_{\mathcal{J}_s\in \mathcal{C}_s}P_{\mathcal{J}_s}(b_{J_1},\ldots,b_{J_s})\prod_{J_j\in \mathcal{J}_s}f_{J_j} 
\end{equation*}
where 
\begin{itemize}
    \item $P_{\mathcal{J}_s}(b_{J_1},\ldots,b_{J_s})=(-1)^{s-1}(b_{J_1}+\cdots+b_{J_s}-1)^{s-2}=(-1)^{s-1}(a+1)^{s-2}$ for every $\mathcal{J}_s$ not compatible with any graph $\Gamma\in G_k$. These are the terms coming from $\int_{\cR_{k-1}}\xi^{n-1-k}$ in \eqref{eqn:minusboundary} which we know by induction;    
    \item $P_{\mathcal{J}_s}(b_{J_1},\ldots,b_{J_s})=-\sum_{\Gamma\in G_k\cap G_{\mathcal{J}_s}}P^{\Gamma}_{\mathcal{J}_s}(b_{J_1},\ldots,b_{J_s})$  for every $\mathcal{J}_s$ compatible with a graph $\Gamma\in G_k$, where $G_{\mathcal{J}_s}$ is the set of graphs $\Gamma$ with which $\mathcal{J}_s$ is compatible. These are the terms coming from $\int_{\cR_{k-1}} \delta_{\Gamma}\xi^{n-2-k}$ in \eqref{eqn:minusboundary}. 
\end{itemize}
For the second case, we introduce a new polynomial 
\begin{equation*}
R_{\mathcal{J}_s}(b_{J_1},\ldots,b_{J_s})=P_{\mathcal{J}_s}(b_{J_1},\ldots,b_{J_s})-(-1)^{s-1}(b_{J_1}+\cdots+b_{J_s}-1)^{s-2}
\end{equation*}
whose degree is at most $s-2$. 

Now the whole proof reduces to show that $R_{\mathcal{J}_s}(b_{J_1},\ldots,b_{J_s})$ is identically zero, for which we will apply induction on the cardinality $s$ of $\mathcal{J}_s$. The induction base case is easy to check and suppose $R_{\mathcal{J}_{s'}}\equiv 0$ for $s' < s$. We will prove $R_{\mathcal{J}_s}\equiv 0$ by showing that each $b_{J_i}$ divides $R_{\mathcal{J}_s}(b_{J_1},\ldots,b_{J_s})$ for $i=1,\ldots, s$. 

A graph $\Gamma\in G_k^c\cap G_{\mathcal{J}_s}$ is compatible with $\mathcal{J}_s$ where the partial sum vanishings $I_{X_1},\ldots, I_{X_m}$ given by the Residue Theorem  on each top component are generated by $\mathcal{I}_{k-1}$. By the global residue condition on $\cR_{k-1}$, the residues of the bottom twisted differential satisfy that $\rho_0=\underline{0}$. Then $N(\mu_0, \underline{0})=0$ as shown in Lemma~\ref{lem:zero} and hence  $P^{\Gamma}_{\mathcal{J}_s}(b_{J_1},\ldots,b_{J_s})=0$ in this case. Therefore, we can drop the requirement of $\Gamma$ belonging to $G_k$ in the summation of $P_{\mathcal{J}_s}$ and rewrite it as 
\begin{equation*}
     P_{\mathcal{J}_s}(b_{J_1},\ldots,b_{J_s})=-\sum_{\Gamma \in G_{\mathcal{J}_s}}P^{\Gamma}_{\mathcal{J}_s}(b_{J_1},\ldots,b_{J_s}). 
\end{equation*}
Moreover, by using the expression of $N(\mu_0,\rho_0)$ in \eqref{eqn:N_0} we have 
\begin{equation*}
    P^{\Gamma}_{\mathcal{J}_s}(b_{J_1},\ldots,b_{J_s}) = \sum_{s'=1}^{S_0}\sum_{\mathcal{J}^0_{s'}\in \mathcal{C}^0_{s'}} P^{\Gamma}_{\mathcal{J}_s,\mathcal{J}^0_{s'}}(b_{J_1},\ldots, b_{J_s})
\end{equation*}
where
\begin{equation*}
    P_{\mathcal{J}_s,\mathcal{J}^0_{s'}}^{\Gamma}(b_{J_1},\ldots, b_{J_s})=(-1)^{s-m+ s'-1}(a+1)^{s'-2}\left(\prod_{J^0_j\in \mathcal{J}^0_{s'}}f_{J^0_j}\right) \left(\prod_{i=1}^m (b_{X_i}-1)^{s_i-1}\right).
\end{equation*}

Now we set the variable $b_{J_1}=0$. Then for every $\Gamma \in G_{\mathcal{J}_s}$ geometrically it means removing the marked poles indexed by $J_1$ from the graph $\Gamma$. The poles indexed by $J_1$ are contained in a single top component, say $X_1$, and for every disjoint collection $\mathcal{J}^0_{s'}=\{J^0_1,\ldots,J^0_{s'}\}$ the node $q_1\in X_1\cap X_0$ is contained in one of these $s'$ partial sum vanishings, say $J^0_1$. Therefore, 
\begin{equation*}
    P_{\mathcal{J}_s,\mathcal{J}^0_{s'}}^{\Gamma}(0,b_{J_2},\ldots, b_{J_s})=K_1^{\Gamma}P^{\Gamma-J_1}_{\mathcal{J}_s-J_1,\mathcal{J}^0_{s'}-J_1}(b_{J_2},\ldots,b_{J_s})
\end{equation*}
where
\begin{equation*}
K^{\Gamma}_1=
    \begin{cases}
        -(b_{X_1}-b_{J_1}-1) & \text{if } J_1 \subsetneq I_{X_1} \hfill (1)\\
        b_{J^0_1}-b_{J_1}-|J^0_1|+1 & \text{if } J_1=I_{X_1} \text{ and } |J^0_1|>1 \hfill (2)\\
        -(a-b_{J_1}+1) & \text{if } J_1=I_{X_1} \text{ and } |J^0_1|=1 \ \ \ \ \ \ \  \hfill (3)
    \end{cases}
\end{equation*}
and $\Gamma - J_1$ represents the graph obtained by removing the poles indexed by $J_1$ from $\Gamma$, the collection $\mathcal{J}_s-J_1=\{J_2,\ldots,J_s\}$, the collection $\mathcal{J}^0_{s'}-J_1=\{J^0_1\setminus\{q_1\},J^0_2\dots,J^0_{s'}\}$ with $b_{J^0_1-J_1}=b_{J^0_1}-b_{J_1}$ if $J_1\subsetneq I_{X_1}$, and otherwise $\mathcal{J}^0_{s'}-J_1 =\{J^0_2,\dots,J^0_{s'}\}$. The value of $K_1^{\Gamma}$ in case (1) is due to that $s_1$ decreases by one, in case (2) the cardinality of $J_1^0$ decreases by one, and in case (3) $s'$ decreases by one.  

Note that cases (1) and (2) cancel out each other except for one term. More precisely, for those graphs satisfying case (2), i.e., if $J_1=I_{X_1} \text{ and } |J^0_1|>1$, then
\begin{equation*}
  b_{J^0_1}-b_{J_1}-|J^0_1|+1=\sum_{q_i\in J^0_1\setminus \{q_1\}}(b_{X_i}-1)
\end{equation*}
which is cancelled with those terms satisfying case (1) and sharing the same graph $\Gamma-J_1$, except for the case when $J_1=I_{X_1}$ and $m=|J^0_1|=2$ because a graph satisfying case (1) and being the same as $\Gamma-{J_1}$ has only one node in the bottom which is not stable. In this exceptional case
\begin{equation*}
    P_{\mathcal{J}_s,\mathcal{J}^0_{s'}}^{\Gamma}(0,b_{J_2},\ldots, b_{J_s})=(-1)^{s-1}(b_{J_2}+\cdots + b_{J_s}-1)^{s-2}.
\end{equation*}
Case (3) for $K^{\Gamma}_1$ only happens if $J_1$ is generated by $\mathcal{I}_{k-1}$ due to the global residue condition on $\cR_{k-1}$. Therefore, if $J_1=I_{X_1}$ and $|J^0_1|=1$, the sum of the polynomials satisfying case (3) is 
\begin{align*}
    \sum_{\substack{\Gamma \in G_{\mathcal{J}_s} \\ J_1=I_{X_1}}} \sum_{s'=2}^{S_0}
    \sum_{\substack{\mathcal{J}^0_{s'}\in \mathcal{C}^0_{s'} \\ J^0_{1}=\{q_1\}}}  & P^{\Gamma}_{\mathcal{J}_s, \mathcal{J}^0_{s'}}(0,b_{J_2},\ldots, b_{J_s}) \\
    & =-(a+1-b_{J_1})\sum_{\substack{\Gamma \in G_{\mathcal{J}_s} \\ J_1=I_{X_1}}} \sum_{s'=2}^{S_0}
    \sum_{\substack{\mathcal{J}^0_{s'}\in \mathcal{C}^0_{s'} \\ J^0_{1}=\{q_1\}}}  P^{\Gamma-J_1}_{\mathcal{J}_s-J_1, \mathcal{J}^0_{s'}-J^0_1}(b_{J_2},\ldots, b_{J_s}) \\ 
    &=(-1)^{s-1}(a+1-b_{J_1})^{s-2}+(a+1-b_{J_1})P_{\mathcal{J}_s-J_1}(b_{J_2},\ldots,b_{J_s})\\
    &=(-1)^{s-1}(a+1-b_{J_1})^{s-2}+(-1)^{s-2}(a+1-b_{J_1})^{s-2}\\
    &=0
\end{align*}
where in the second equality the term $(-1)^{s-1}(a+1-b_{J_1})^{s-2}$ comes from the case $s'=2$ and in the third equality we use the induction hypothesis as $|\mathcal J_s - J_1 | < |\mathcal J_s|$. Note that for $s'>2$, the graph $\Gamma - J_1$ is compatible with the collection $J_s-J_1$ and $J^0_{s'}-J_1=\{J^0_2,\ldots, J^0_{s'}\}$. Then the collection of these graphs consists of all graphs that are compatible with $\mathcal{J}_s-J_1$. We thus conclude that
\begin{equation*}
    P_{\mathcal{J}_s}(0,b_{J_2},\ldots, b_{J_s})=(-1)^{s-1}(b_{J_2}+\cdots + b_{J_s}-1)^{s-2}
\end{equation*}
and then
\begin{equation*}
    R_{\mathcal{J}_s}(0,b_{J_2},\ldots, b_{J_s})=0.
\end{equation*}

The above implies that $b_{J_1}$ divides $R_{\mathcal{J}_s}$ as a polynomial of $b_{J_1},\ldots, b_{J_s}$. By symmetry $b_{J_1}\cdots b_{J_s}$ divides $R_{\mathcal{J}_s}$. Since the degree of 
$R_{\mathcal{J}_s}$ is bounded by $s-2$, we obtain that $R_{\mathcal{J}_s}\equiv 0$ as desired. 
\end{proof}

As a biproduct of the preceding proof, we can verify the combinatorial properties of $N(\mu, \rho)$ stated in Corollary~\ref{cor:numerical}. 

\begin{proof}[Proof of Corollary~\ref{cor:numerical}]
We first prove the inequality. By induction on the number of independent partial sum vanishings, it suffices to consider the case where $\rho'$ satisfies one more independent partial sum vanishing condition than $\rho$, say, indexed by $I$. Let $\Gamma$ be the two-level graph with exactly two top level components containing the marked poles in $I$ and in $I^c$ respectively. Note that the residues of $\rho$ in $I^c$ cannot be all zero, as otherwise $I$ would be generated by the partial sum vanishings of $\rho$ modulo the total sum vanishing. Therefore, $\Gamma$ makes a nonzero contribution to the boundary correction terms in \eqref{eqn:minusboundary}, which implies that $N(\mu, \rho) > N(\mu, \rho')$. 

Next for the degree of $N(\mu, \rho)$, we have already seen that it is bounded by $n-2$. To show the degree equal to $n-2$ whenever $\rho \neq \underline{0}$, assume that the claim holds when the number of poles is less than $n$. Now for the case of $n$ poles, let $\rho_0, \rho_1, \ldots, \rho_{n-2}, \rho_{n-1} = \underline{0}$ be a sequence of residue tuples where $\rho_0$ is general and each $\rho_k$ satisfies exactly one more independent partial sum vanishing than $\rho_{k-1}$, indexed by $I_k$, for $k=1,\ldots, n-1$. Consider $\rho_{n-2}$ first. Let $\Gamma$ be the two-level graph with exactly two top level components containing the marked poles in $I_{n-2}$ and in $I^c_{n-2}$ respectively. Note that $N(\mu, \rho_{n-1}) =  N(\mu, \underline{0}) = 0$ and by induction $\Gamma$ makes a polynomial contribution of degree $(|I_{n-2}| - 2) + (|I_{n-2}^c| - 2) + 2 = n-2$ in the boundary correction terms in \eqref{eqn:minusboundary} by using the expression \eqref{eqn:product} where the extra degree two comes from the twist $t_\Gamma$ at the two nodes of $\Gamma$ (and the same degree holds for any other nonzero boundary correction terms). Therefore, the degree of $N(\mu, \rho_{n-2})$ is exactly $n-2$. Since the degree of $N(\mu, \rho_k)$ is bounded by $n-2$ and $N(\mu, \rho_k) > N(\mu, \rho_{n-2})$ for $k < n-2$, the degree of $N(\mu, \rho_k)$ must also be $n-2$. 
\end{proof}

As mentioned in the introduction we provide a direct and short proof for the special case of all but two residues being zero.  

\begin{prop}
\label{prop:k=n-2}
If $\rho = (0, \ldots, 0, r, -r)$ where $r\neq 0$, then 
\begin{align}
\label{eqn:k=n-2}
 N(\mu, \rho) &= \int_{\cP_n(\mu)} \prod_{i=1}^{n-2} (b_i-1) \psi_{p_i} \\ \nonumber 
&=(n-2)! \prod_{i=1}^{n-2}(b_i-1)  \\ \nonumber 
& = \sum_{s=1}^{n-1} (-1)^{s-1}(a+1)^{s-2}\sum_{\substack{J_1\sqcup \cdots \sqcup J_s = \{1,\ldots, n\}\\ \{n-1, n\}\subset J_s}} \prod_{j=1}^s f_{J_j}
 \end{align}
 where the last summation runs over all partitions of the set of poles such that the two poles with nonzero residues are contained in the same part.  
\end{prop}

\begin{proof}
Denote by $\cR_i$ the closure of the locus in $\cP(\mu)$ parameterizing differentials that have zero residue at $p_i$.  By \cite[Propositions 8.2 and 8.3]{CMZ} the divisor class of $\cR_i$ is given by 
$$ \cR_i = (b_i - 1) \psi_{p_i}. $$
As in the proof of Theorem~\ref{thm:main}, one checks that the intersection $\cR_1\cap \cdots \cap \cR_{n-2}$ is transversal and consists of interior points only, which implies the first equality in \eqref{eqn:k=n-2}. The second equality thus follows from \eqref{eqn:psi}. Applying \eqref{eqn:main} to this case yields the last equality. 

We remark that the last summation expression in \eqref{eqn:k=n-2} does not seem to match the simplified expression out of the $\psi$-product on the nose. Indeed this provides an interesting combinatorial identity analogous to Lemma~\ref{lem:zero}, which can also be verified directly. Denote by $P(b_1,\ldots,b_n)$ the last summation in \eqref{eqn:k=n-2}, which we know is a polynomial of degree bounded by $n-2$. For each pole $p_i$ with zero residue for $i=1,\ldots,n-2$, the same proof as in Lemma~\ref{lem:zero} shows that $b_i-1$ divides $P(b_1,\ldots,b_n)$. Therefore,   $P(b_1,\ldots,b_n)=c\prod_{i=1}^{n-2}(b_i-1)$ for some constant $c$. To determine $c$, we set $b_1=\cdots=b_{n-1}=0$, $b_n=1$, and hence $a=-1$. On the one hand, the polynomial expression gives $P(0,\ldots,0,1)=(-1)^{n-2}c$. On the other hand, we can directly plug in the summation expression and obtain that
\begin{align*}
     P(0,\ldots, 0,1) &= f(-1, n) - \sum_{\substack{J_1\sqcup J_2 = \{1,\ldots,n\}\\ \{n-1, n\}\subset J_2}} f_{J_1} f_{J_2} \\
     & = (-1)^{n-2} (n-2)! - \sum_{\substack{J_1\sqcup J_2 = \{1,\ldots,n\}\\ \{n-1, n\}\subset J_2}} f(-1, |J_1|+1) \cdot f(0, |J_2|+1). 
\end{align*}
Note that $|J_2| \geq 2$ implies 
$f(0, |J_2|+1) = 0\cdot (-1) \cdots (-|J_2|+2) = 0$. Therefore, $P(0,\ldots, 0,1) = (-1)^{n-2} (n-2)!$ which implies that the constant $c = (n-2)!$, and hence $P(b_1,\ldots,b_n)=(n-2)!\prod_{i=1}^{n-2}(b_i-1)$ as desired. 
\end{proof}

\end{document}